\def\bbuildrel#1_#2^#3{\mathrel{\mathop{\kern 0pt#1}\limits_{#2}^{#3}}}

\def\do{\hbox to 20pt{\rightarrowfill}}

\def\NN{\mathbb N}

\def\RR{\mathbb R}
\def\QQ{\mathbb Q}

\def\Pot{\hbox{$\mathcal P$}}

\def\bs{\bigskip}
\def\ms{\medskip}

\def\w{\thinspace\hbox{\hsize 14pt \rightarrowfill}\thinspace}

\def\0{\hbox{$\emptyset$}}
\def\A{\hbox{$\mathcal A$}}
\def\I{\hbox{$\mathcal I$}}

\def\s{\hbox{$\sigma$}}

\def\sub{\subseteq}

\def\M{\hbox{$\mathcal M$}}
\def\N{\hbox{$\mathcal N$}}
\def\F{\hbox{$\mathcal F$}}

\def\C{\mathscr{C}}
\def\K{\mathscr{K}}

\def\M{\mathscr{M}}

\def\D{\mathscr{D}}
\def\S{\mathscr{S}}

\def\Ra{\Rightarrow}

\def\supp{\rm supp}

\documentclass[a4paper,12pt]{amsart}

\usepackage[final]{changes} 

\usepackage{amssymb}
\usepackage{mathrsfs}

\usepackage{mathabx}
\newcommand{\cantor}{2^{\NN}}

\newtheorem{theorem}{Theorem}[section]
\newtheorem{lemma}[theorem]{Lemma}

\newtheorem{proposition}[theorem]{Proposition}
\newtheorem{remark}[theorem]{Remark}

\numberwithin{equation}{section}
\newtheorem{claim}[theorem]{Claim}

\begin{document}

\title{On Mazurkiewicz's sets, 
	thin \s-ideals of compact sets
	and the space of probability measures on the rationals}

\author{R. Pol and P. Zakrzewski}
\address{Institute of Mathematics, University of Warsaw, ul. 
02-097 Warsaw, Poland}
\email{pol@mimuw.edu.pl, piotrzak@mimuw.edu.pl}





\begin{abstract}
 We shall establish some properties of thin $\sigma$-ideals of compact sets in compact metric spaces (in particular, the $\sigma$-ideals of compact
null-sets for thin subadditive capacities), and we shall refine the celebrated theorem of David Preiss that there exist compact non-uniformly tight sets of probability measures on the rationals.

Both  topics will be based on a construction of Stefan Mazurkiewicz from his 1927 paper containing a solution of a Urysohn's problem in dimension theory. 
\end{abstract}

\maketitle

\section{Introduction}\label{introduction}

The Mazurkiewicz sets appeared in \cite{M} as a key element of his solution of a fundamental problem of Urysohn. Embedded in the dimension theory context, this brilliant construction was subsequently somewhat forgotten.

Years later, some of its variations were rediscovered in different settings, demonstrating its usefulness in other areas of topology and measure theory.

In Section \ref{Mazurkiewicz}, we recall the original Mazurkiewicz construction, and  present some of its modifications, suitable for our purposes. 

We shall use Mazurkiewicz's sets in two ways.

Firstly, we shall consider regularity properties 
of capacities and Borel measures on a compactum (i.e., a compact metrizable space) $X$.
Let us recall that a Borel measure $\mu$ on  $X$ is  {\sl semifinite} if each Borel set of positive $\mu$-measure contains a Borel set of finite positive $\mu$-measure.
Let us also recall that a capacity on $X$ is {\sl thin} if there is no uncountable family  of pairwise disjoint compact subsets of $X$ of positive capacity, cf. \cite{k-l-w}.

 We denote by $K(X)$ the space of compact subsets of $X$, equipped with the Vietoris topology, cf. \cite{k}.


\smallskip
   
We shall establish the following two theorems. In fact, we shall discuss this topic in a more general setting concerning \s-ideals of compact sets, and the theorems will be derived from some results about \s-ideals, cf. Section \ref{thin}.
 
\begin{theorem}\label{thin measures}
	Let  $X$ be  a compactum. Let $\mu$ be a non-atomic semifinite Borel measure on $X$ such that the collection of compact $\mu$-null sets is coanalytic in $K(X)$.
		
	 If $\mu$ is not \s-finite, then there is a $G_{\delta\sigma}$-set $M$ in $X$ such that:
	  \begin{enumerate}
		\item[(i)]  $\mu(G \setminus M)=0$ for no $G_\delta$-set $G$ in $X$ containing $M$,
	 	\item[(ii)]  there is a semifinite Borel measure $\mu'\ll\mu$ on $X$ such that $\mu'(M)=0$ but $M$ is not contained in any $\mu'$-null $G_\delta$-set in $X$.
	 
\end{enumerate}

\end{theorem}
	
	\begin{theorem}\label{thin capacities}
		Let  $X$ be  a compactum. Let $\gamma$ be a non-atomic subadditive capacity on $X$.	
		
		If $\gamma$ is not thin, then there is a $G_{\delta \sigma}$-set $M$ in $X$ such that:
		\begin{enumerate}
			\item[(i)]  $\gamma(G \setminus M)=0$ for no $G_\delta$-set $G$ in $X$ containing $M$,
			\item[(ii)] there is a subadditive capacity $\gamma'\ll\gamma$ on $X$  such that $\gamma'(M)=0$ but $M$ is not contained in any $\gamma'$-null $G_\delta$-set in $X$. 
				
	\end{enumerate}
		
\end{theorem}

Secondly, we shall consider the space $P(\QQ)$ of probability measures on the rationals, equipped with the weak topology. We say that a subset $A$ of $P(\QQ)$ is 
{\sl \s-uniformly tight} if it covered by countably many uniformly tight sets, cf. \cite{Bo}, \cite{Bo-2}. If $\mu\in P(\QQ)$, then by 
${\rm supp} (\mu)$ we denote the support of $\mu$, i.e., the closure in $\QQ$ of the set $\{q\in \QQ: \mu(\{q\})>0 \}$.

We shall refine the celebrated theorem of David Preiss \cite{P} (cf. \cite[Theorem 4.8.6]{Bo-2}) that the space $P(\QQ)$ 
contains a compact, non-uniformly tight set, to the following effect.

\begin{theorem}\label{main}
There exists a compact nonempty set $K$ in $P(\QQ)$ such that 

\begin{enumerate}
	
	\item[(i)] ${\rm supp} (\mu)$ is locally compact for $\mu\in K$, and $\supp(\mu)\cap \supp(\nu)$ is finite for distinct $\mu,\nu\in K$,
	
	\item[(ii)] any nonempty open set $V$ in $K$ contains a compact set $L$ such that, whenever $A\sub K$ is \s-uniformly tight, $L\setminus A$ contains a non-uniformly tight compact set.
	
\end{enumerate}
\end{theorem}


Using a reasoning from \cite{PZ}, we shall also show that the set $K$ in Theorem \ref{main} has the following property (a weaker version of the ``1-1 or constant property", introduced by Sabok and Zapletal \cite{SZ}): any Borel function $f: K\w [0,1]$ is either constant or injective on a Borel non-\s-uniformly tight set in $K$.

\ms

Proofs of Theorems \ref{thin ideals} and \ref{main} will be given in Sections \ref{thin} and \ref{P(Q)}, respectively, and the result concerning Borel maps on $P(\QQ)$, stated above, will be addressed in Section \ref{1-1}. 

In comments, gathered in Section \ref{Comments}, we shall provide some information, and natural questions, concerning the \s-ideal generated by compact uniformly tight sets in $P(\QQ)$. 

\bs


\section{Mazurkiewicz's sets}\label{Mazurkiewicz}

\subsection{The Mazurkiewicz construction}\label{construction}

The following construction appeared in \cite[Sections 6, 7 and 8]{M}, cf. Remark \ref{remark}(B). 

Let $f:X\w Y$ be a continuous surjection of the compactum $X$ onto an uncountable compactum $Y$, and let $T$ be a copy of the Cantor set in $Y$.

 Let
\begin{enumerate}
	\item[(1)] $\F=\{(F_1, F_2,\ldots)\in K(X)^{\NN}:
	F_1\sub F_2\sub\ldots  \} .$
\end{enumerate}

 Since $\F$ is compact, as a closed subspace of the product $K(X)^{\NN}$, there is a continuous surjection
\begin{enumerate}
	\item[(2)] $t\mapsto (F_1(t), F_2(t),\ldots)$ from $T$ onto $\F$.
	
\end{enumerate}





Let us notice that the mapping $t\mapsto F_k(t)\cap f^{-1}(t)$ from $T$ into $K(X)$
 is upper semi-continuous. 
 It follows that sets
\begin{enumerate}
	\item[(3)] $D_0=\emptyset$ and $D_k=\{t\in T: F_k(t)\cap f^{-1}(t)\neq\emptyset\}$ for $k=1, 2, \ldots$
\end{enumerate}
are compact. 

Let 
\begin{enumerate}
	\item[(4)] $D=\bigcup_{k=0}^\infty D_k$. 
\end{enumerate}

Let us notice that $D_0\sub D_1\sub\ldots$, cf. (1), and 
$D_1\neq \emptyset$, as for $t$ such that $X=F_1(t)=F_2(t)=\ldots$ we have $t\in D_1$.

Mazurkiewicz proved in Section 6 of \cite{M} a selection theorem which provided, for each  $k\geq 1$, a  Baire class 1 function $\varphi_k: D_k\w X$ such that, cf. \ref{remark}(A),

\begin{enumerate}
	\item[(5)] 	
$\varphi_k(t)\in F_k(t)\cap f^{-1}(t)$ for $t\in D_k$. 
\end{enumerate}

Mazurkiewicz's set $M$ is finally defined  as follows, cf. \cite[Section 8]{M}:
\begin{enumerate}
	\item[(6)] 
	$M=\bigcup_{k\geq 1} \varphi_k(D_k\setminus D_{k-1}).$
\end{enumerate}


The  set $M$ has the following property, where by a partial selector for $f:X\w Y$ we understand a subset of  $X$ intersecting every fiber of $f$ in at most one point:

\begin{enumerate}
	\item[(M)] {\sl $M$ is a $G_{\delta\sigma}$-set in $X$ which is a partial selector for $f$ and each $G_\delta$-set in $X$ containing $M$ contains also some fiber $f^{-1}(y)$.}
\end{enumerate}


To see this, first let us note that $\varphi_k(D_k\setminus D_{k-1})$ is a $G_\delta$-set in $X$ for each $k\geq 1$. Indeed, if $G_k$ is the graph of  $\varphi_k$, then $G_k$ is a $G_\delta$-set as  the graph of a Baire class 1 function and, moreover, cf. (5),   
$$\varphi_k(D_k\setminus D_{k-1})=\{x\in X: (f(x),x)\in G_k \} \setminus f^{-1}(D_{k-1}).$$ Consequently,  cf. (6), $M$ is a $G_{\delta\sigma}$-set in $X$.

	Next, aiming at a contradiction, assume that $H$ is a $G_\delta$-set in $X$ containing $M$ but no fiber of $f$. Then $X\setminus H= F_1\cup F_2 \cup\ldots$ for some $(F_1, F_2,\ldots)\in \F$, cf. (1),
	 such that $f(\bigcup_k F_k)=Y$. It follows that letting $t\in T$ be such that $F_k=F_k(t)$ for $k=1,2,\ldots$, cf. (2), we have that $t\in D$, cf. (4).  
	 Let $k$ be such that $t\in D_k\setminus D_{k-1}$. Then $\varphi_k(t)\in M\setminus H$, cf. (5), which is impossible.

\subsubsection{Remark}\label{remark}

(A) The selection theorem, established in \cite[Section 6]{M}, providing a Baire class 1 selector for any upper semi-continuous mapping defined on a metric space and taking closed non-empty subsets of a Polish space as values, was rediscovered in \cite{Bou} and became a standard tool in the descriptive set theory, cf. \cite[Theorem XIV.4]{ku-mo}.

\smallskip

(B) To be more accurate, constructing his set in \cite{M}, Mazurkiewicz considered as $X$ the closed unit ball in $\RR^n$ and the function $f:X\w [0,1]$ assigning to  each $x\in X$ its distance from the origin. The property $(M)$ was used by Mazurkiewicz to establish that the set $M$ has dimension $n$ (cf. also \cite{Pol}).

 \subsection{Special Mazurkiewicz sets}\label{special}

To get Theorem \ref{main}, we shall need a
special adjustment
 of the Mazurkiewicz construction. Before giving the details, let us make some introductory remarks, adopting the notation from the preceding section. The set
\begin{enumerate}
	\item[] 
$\F_0=\{(F_1, F_2,\ldots)\in \F:
	F_1= F_2=\ldots\hbox{ and } f(F_1)=Y  \}$
\end{enumerate}
is compact, and so is the set
\begin{enumerate}
		\item[] 
	$T_0=\{t\in T: (F_1(t), F_2(t),\ldots)\in \F_0\}$.
\end{enumerate}
Moreover, cf. (3), $T_0\sub D_1$ and hence $M\cap f^{-1}(T_0)=\varphi_1(T_0)$ is a $G_\delta$-partial selector of $f$, hitting each compactum in $X$ which is mapped by $f$ onto $Y$.

This part of the Mazurkiewicz set was rediscovered by Michael \cite{Mi} (with a similar justification), while investigating compact-covering mappings, and it was used by Davies \cite{D} to provide a striking example concerning uniform tightness of collections of measures (Davies overlooked in \cite{D} the Michael's paper and gave a direct construction of such sets in a special case, cf. also \cite{D2}).

A key element of our proof of Theorem \ref{main} will be a refinement of the Davies example, based on the following special instance of the Mazurkiewicz construction.

Let $\pi:\cantor\times\cantor\w \cantor$ be the projection onto the first axis, and let
\begin{enumerate}
	\item[(6)] 
	$\C=\{C\in K(\cantor\times\cantor): \pi(C)=\cantor\}$
\end{enumerate}
(this set can be identified with $\F_0$, where $X=\cantor\times\cantor$ and $f=\pi$).

Since $\C$ is a compact zero-dimensional set without isolated points, it is a Cantor set, and parametrizing $\C$ on $\cantor$, we can demand that 
\begin{enumerate}
	\item[(7)] 
	$h:\cantor\w \C$ is a homeomorphism.
\end{enumerate}

Then we let, cf. \cite[Proof of Lemme 5]{M},
\begin{enumerate}
	\item[(8)] 
	$\s(t)=\min(h(t)\cap \pi^{-1}(t))$ and $M=\s(\cantor)$,
\end{enumerate}
where the minimum is taken with respect to 
the lexicographical ordering on $\cantor$ (cf. \cite[2D]{k}).

Let us notice that $\s$ is a Baire class 1 function with the property that $(\pi\circ \sigma)(t)=t$ for $t\in \cantor$. Consequently,  $M$ is a $G_\delta$-set (cf. the argument following assertion (M) in Section \ref{construction}).

We define
\begin{enumerate}
	\item[(9)] 
	$T(C)=\{t\in\cantor: h(t)\sub C \} $, for $C\in \C$.
	
\end{enumerate}

Since $\{F\in \C: F\sub C \}$ is compact, so is $T(C)$.

\begin{lemma}\label{T(C)}\hfill\null
\begin{enumerate}
	\item[(A)] For each $C\in \C$, $T(C)\sub \pi(M\cap C)$ and 
	$T(C')\sub T(C)$, whenever $C'\sub C$ belongs to $\C$.
	
	\item[(B)] For each nonempty open set $V$ in $\cantor$, there is $C\in\C$ such that $C$ is a finite union of closed-and-open rectangles in $\cantor\times\cantor$ and $T(C)\sub V$.
		   
\end{enumerate}

\end{lemma}

\begin{proof}
	(A) 
	If $t \in T(C)$, then $\s(t)\in M\cap C$, cf. (8) and (9), and hence $t=\pi(\s(t))\in \pi(M\cap C)$.
	
	
	\smallskip
	
	(B) Since $h$ is a homeomorphism onto $\C$, $h(V)$ is open in $\C$, and hence there are closed-and-open sets $U_1,\ldots,U_m$ in $\cantor\times\cantor$ such that whenever $F\in\C$ intersects all $U_i$ and $F\sub\bigcup_{i=1}^m U_i$, then $F\in h(V)$. 
	
	Moreover, one can assume that $U_i$ are closed-and-open rectangles in  $\cantor\times\cantor$ and the projections $\pi(U_i)$, $\pi(U_j)$ are either identical or disjoint.
	
	Let $\S$ be the collection of projections $\pi(U_i)$, $i=1,\ldots,m$. Let us fix $S\in\S$ and let $U_{i_1},\ldots,U_{i_k}$ be the rectangles $U_i$ with $\pi(U_i)=S$. Let us split $S$ into pairwise disjoint closed-and-open sets $S_1,\ldots,S_k$ and let us replace each rectangle  $U_{i_j}$ by the rectangle $W_{i_j}=U_{i_j}\cap (S_j\times \cantor)$, $j\leq k$.
	
	Let $C=\bigcup_i\bigcup_j W_{i_j}$. Since $h(V)$ is nonempty, we have $\bigcup\S=\cantor$. But $\pi(C)=\bigcup\S$, so consequently $C\in\C$, cf. (6). Let $F\in\C$ and $F\sub C$. Since the projections $\pi(W_{i_j})$ are pairwise disjoint, $F$ intersects each $W_{i_j}\sub U_i$. Also, $F\sub\bigcup_{i=1}^m U_i$, and hence $F\in h(V)$. Since $h$ was injective, we infer that $T(C)\sub V$, cf. (9).  
	
	\end{proof}

\section{On thin \s-ideals of compact sets}\label{thin}

Most of our notation and terminology in this section follow  \cite{k-l-w}.

Given a subset $E$ of a compactum $X$
we denote by $Bor(E)$ 
the \s-algebra of Borel sets in $E$.

A collection $I\sub K(X)$ is  {\sl hereditary} 
if it is closed under taking compact subsets of its elements. If  $I$ is hereditary and, moreover, closed under compact countable unions of elements of $I$,
then $I$ is {\sl a \s-ideal of compact sets in $X$}.

A {\sl \s-ideal 
	on $X$} is a collection $J\sub Bor(X)$, closed under taking Borel subsets and countable unions of elements of $J$.   We always assume that $X\notin J$.

Let $I$ be a 
\s-ideal of compact sets in a compactum $X$.

We let
$\tilde{I}_{Bor(X)}= \{B\in Bor(X): K(B)\sub I \}$.
The collection $\tilde{I}_{Bor(X)}$ 
has {\sl the inner approximation property}, namely every Borel set not in $\tilde{I}_{Bor(X)}$ contains a compact subset not in $\tilde{I}_{Bor(X)}$. Conversely, if  $J$ is a \s-ideal on $X$ with the inner approximation property and we let $I=J\cap K(X)$, then $I$ is a \s-ideal of compact sets in  $X$ and  $J=\tilde{I}_{Bor(X)}$.


We say that $I$ is {\sl thin} if there is no uncountable disjoint family  of compact subsets of $X$ which are not in $I$, cf. \cite{k-l-w}. If $\tilde{I}_{Bor(X)}$ is a \s-ideal on $X$, then $I$ is thin if and only if $\tilde{I}_{Bor(X)}$ satisfies c.c.c. 

It is well-known that if  a coanalytic
\s-ideal $I$ of compact sets in a compactum $X$ is not thin, then there is a Cantor set of pairwise disjoint  compact sets not in $I$ (cf. \cite[Section 3.1, Theorem 2]{k-l-w}). Combining this fact with properties of the Mazurkiewicz set we obtain the following observation.

\begin{lemma}\label{main lemma}
	Let $I$ be a coanalytic
	\s-ideal of compact sets in a compactum $X$. If $I$ is not thin, then there is a continuous map $\Phi:\cantor\w K(X)$ and a $G_{\delta\sigma}$-set $M\sub \bigcup \{\Phi(t): t\in\cantor \}$ such that 
	
	\begin{enumerate}
		
		\item[(i)] the compact sets $\Phi(t)$ are pairwise disjoint and   not in $I$,
		
		\item[(ii)] $|M\cap \Phi(t)|\leq 1$ for each $t$,
		
		\item[(iii)] for any $G_\delta$-set $G$ in $X$ containing $M$ there is $t$ with $\Phi(t)\sub G$.
		
	\end{enumerate}
	
\end{lemma}

\begin{proof}
	As recalled above, there is a continuous map $\Phi:\cantor\w K(X)\setminus I$ such that the sets  
	$\Phi(t)$ are pairwise disjoint. 
	
	Let $\K=\{\Phi(t): t\in\cantor \}$, $K=\bigcup \K$ and 
	let $s:K\w \K$ associate to each $x\in K$ the unique $L_x\in\K$ such that $x\in L_x$.
	
	Clearly, both $\K$ and $K$ are compact in the respective spaces and 
	the mapping $s$ is continuous, as for each compact set $\A$ in $\K$, $s^{-1}(\A)=\bigcup\A$ is compact in $K$.
	
	
	It follows that the surjection $f=\Phi^{-1}\circ s:K\w \cantor$, associating to each $x\in K$ the unique $t\in\cantor$ such that $x\in \Phi(t)$,  is also continuous. 
	
	Finally, let $M$ be the Mazurkiewicz set for $f$ (cf. assertion (M) in Section \ref{Mazurkiewicz}). Clearly, $M$ satisfies conditions (ii) and (iii).   
	
\end{proof}

We shall obtain Theorems \ref{thin measures} and \ref{thin capacities} by specifying (to the $\sigma$-ideals of compact null sets with respect to measures and capacities) the following general theorem concerning arbitrary coanalytic \s-ideals of compact sets. It complements earlier results of Kechris, Louveau and Woodin (cf. \cite[Section 3.1, Theorem 7]{k-l-w}), concerning the relationship between thinnes of \s-ideals of compact sets and their regularity properties.

\begin{theorem}\label{thin ideals}
	Let $I$ be a coanalytic
	\s-ideal of compact sets in a compactum $X$. If 
	$\tilde{I}_{Bor(X)}$ is a \s-ideal on $X$ containing all singletons, then the following are equivalent:
	
	\begin{enumerate}
		
		\item[(a)] $I$ is thin,
		
		\item[(b)] If $J\supseteq \tilde{I}_{Bor(X)}$ is any \s-ideal, then every set in $J$ is contained in a $G_\delta$-set in $J$.
		
	\end{enumerate}
	
	Moreover, if $I$ is not thin, then there is a  $G_{\delta\sigma}$-set $M$ in $X$ such that:
	\begin{enumerate}
		\item[(i)]  $G \setminus M\in \tilde{I}_{Bor(X)}$ for no $G_\delta$-set $G$ in $X$ containing $M$,
		\item[(ii)]  there is a \s-ideal  $I'\supseteq \tilde{I}_{Bor(X)}$ on $X$  such that  $M\in I'$ but $M$ is not contained in any $G_\delta$-set from $I'$. 
		
	\end{enumerate}
	 
\end{theorem}








\begin{proof}
	Assume first that $I$ is not thin. Let  a continuous map $\Phi:\cantor\w K(X)$ and a $G_{\delta\sigma}$-set $M\sub \bigcup \{\Phi(t): t\in\cantor \}$ satisfy assertions of Lemma \ref{main lemma}. 
	
	Condition (i) is clearly satisfied.
	
	To obtain condition (ii) let
	\begin{enumerate}
		\item[(1)]  $I'=\{A\in Bor(X):\ A\cap \Phi(t)\in \tilde{I}_{Bor(X)} \hbox{\ for each\ } t\in\cantor    \}$.
	\end{enumerate}	
	
	Clearly, $I'$ is a \s-ideal on $X$, $\tilde{I}_{Bor(X)}\sub I'$ and $M\in I'$ 
	(each $M\cap \Phi^{-1}(t)$ having at most one element).
	 Moreover, for any $G_\delta$-set $G$ in $X$ containing $M$ there is $t$ with $\Phi(t)\sub G$ (see Lemma \ref{main lemma}) and hence $G\notin I'$. 
	
	We have proved that $(b) \Ra (a)$ and the ``moreover" part of the assertion.
	
	\smallskip
	
	Assume now that $I$ is thin and let $J\supseteq \tilde{I}_{Bor(X)}$ be a \s-ideal on $X$.  Since the \s-ideal 
	$\tilde{I}_{Bor(X)}$ is c.c.c., so is $J$. 
	
	We claim that $J$ has the inner approximation property.
	
	Indeed, $\tilde{I}_{Bor(X)}$ being c.c.c.,	letting $C= X\setminus \bigcup {\mathcal R}$, where $\mathcal R$
	is a maximal disjoint family of Borel sets from $J \setminus \tilde{I}_{Bor(X)}$, we get  a Borel set 
	$C$ such that  
	
	\begin{enumerate}
		\item[(2)]	$J=\{B\in Bor(X): B\cap C \in\tilde{I}_{Bor(X)}\}$.
	\end{enumerate}

	Now, if $B\notin J$, then $B\cap C \notin \tilde{I}_{Bor(X)}$, so by the inner approximation property
	of $\tilde{I}_{Bor(X)}$, there is a compact set $K\sub B\cap C$ not in $\tilde{I}_{Bor(X)}$ and hence also not in $J$ (cf. (2)).
	
	Finally, it is enough to note that the inner approximation property plus c.c.c. implies that every set  $B\in J$ is contained in a $G_\delta$-set $G$ in $J$. To see this, let us just consider a maximal family $\A$ of pairwise disjoint and disjoint from $B$ compact sets not in $J$ and let $G=X\setminus \bigcup\A$.
	
	We have thus proved  that $(a) \Ra (b)$, completing the proof of the theorem.

\end{proof}

Natural examples of \s-ideals with the inner approximation property on a compactum $X$ are the \s-ideals of Borel null sets with respect to semifinite Borel measures or capacities on $X$. 

Let us recall that that a Borel measure $\mu$ on  $X$ is  {\sl semifinite} if each Borel set of positive $\mu$-measure contains a Borel set of finite positive $\mu$-measure (\s-finite Borel measures and Hausdorff measures on Euclidean cubes are semifinite, cf. \cite{r}). If $\mu$ is such a measure, then we let ${I}_\mu=\{K\in K(X): \mu(K)=0\}$ and
$\tilde{I}_\mu=\{B\in Bor(X): \mu(B)=0\}$. This \s-ideal is c.c.c. if and only if the measure $\mu$ is \s-finite. The inner approximation property of $\tilde{I}_\mu$ follows from the inner regularity of finite Borel measures on Polish spaces (see \cite[Theorem 17.11]{k}). 

By a {\sl capacity on $X$} we mean here a  function 
$\gamma:\Pot(X)\w [0,+\infty)$ such that (cf. \cite[Section 3.1]{k-l-w}) 
\begin{enumerate}
	\item $\gamma(\emptyset)=0$ and $A\sub B$ implies $\gamma(A)\leq \gamma(B)$,
	\item $\gamma(\bigcup_n A_n)=\sup_n \gamma(A_n)$, if $A_0\sub A_1\sub A_2\sub\ldots$,
	 \item $\gamma(\bigcap_n K_n)=\inf_n \gamma(K_n)$, if $K_0\supseteq K_1\supseteq K_2\supseteq\ldots$ are compact sets.
\end{enumerate}

If  a capacity $\gamma$ on $X$ is a {\sl subadditive} (i.e., $\gamma(A\cup B)\leq \gamma(A)+\gamma(B)$, whenever $A,\ B\sub X$, cf. \cite[Section 3.1]{k-l-w}),  then the collection ${I}_\gamma=\{K\in K(X): \gamma(K)=0\}$ is a \s-ideal of compact sets and $\tilde{I}_\gamma=\{B\in Bor(X): \gamma(B)=0\}$ is a \s-ideal on $X$.  If this \s-ideal  is c.c.c., then we say that $\gamma$ is {\sl thin}. The inner approximation property of $\tilde{I}_\gamma$ follows from
the fundamental Choquet capacitability theorem (see \cite[Section 3.1]{k-l-w}). 

For \s-ideals of compact sets of the form ${I}_\mu$ and ${I}_\gamma$ 
the ``moreover" part of Theorem \ref{thin ideals} 
is specified by Theorems \ref{thin measures} and \ref{thin capacities} which we are now ready to prove.

\begin{proof}[Proof of Theorems \ref{thin measures} and \ref{thin capacities}]
	We shall closely follow the first part of the proof of Theorem \ref{thin ideals} letting $I={I}_\mu$ ($I={I}_\gamma$, respectively; let us note that in this case $I$  is always a $G_\delta$-set in $K(X)$, see \cite[Section 3.1]{k-l-w}) in which case we have $\tilde{I}_{Bor(X)}=	\tilde{I}_\mu$ 
	($\tilde{I}_{Bor(X)}=\tilde{I}_\gamma$, respectively). 
	
	In both cases it is enough to show that the \s-ideal $I'$ (cf. (1)) is of the form $\tilde{I}_{\mu'}$ ($\tilde{I}_{\gamma'}$, respectively) for a certain 
	semifinite Borel measure $\mu'$ (subadditive capacity $\gamma'$, respectively). We achieve this by defining $\mu'$ and $\gamma'$ by the formulas:
	
	$$\begin{matrix}
	
	\mu'(A) & = & \sum_t \mu(A\cap \Phi(t)) &\quad\hbox{for}\quad A\in Bor(X),\\
	
	\gamma'(A) & = & \sup_t\gamma(A\cap \Phi(t)) &\quad\hbox{for}\quad A\sub X.
\end{matrix}
$$

Checking all the required properties of $\mu'$ and $\gamma'$ is straightforward, except perhaps for property (3) of $\gamma'$ which boils down to 
\begin{enumerate}
	\item[(4)]	$\sup_t\inf_n\gamma(K_n\cap \Phi(t))\geq \inf_n\sup_t\gamma(K_n\cap \Phi(t))$, if $K_0\supseteq K_1\supseteq K_2\supseteq\ldots$ are compact sets.
\end{enumerate}	

In order to prove it, let $a= \inf_n\sup_t\gamma(K_n\cap \Phi(t))$ ($a\in [0,+\infty)$) and let us fix an arbitrary real number $b<a$. Then for each $n$ we may choose $t_n\in\cantor$ such that $\gamma(K_n\cap \Phi(t_n))\geq b$ and the sequence $(t_n)_n$ is convergent in $\cantor$ to $t'$.
Let $K=\bigcap_n K_n$.

 We claim that $\gamma(K\cap \Phi(t'))\geq b$; it then follows that $\gamma(K_n\cap \Phi(t'))\geq b$ for each $n$, which since $b<a$ was arbitrary, completes the proof of (4).
 
 To justify the claim, suppose towards a contradiction that   $\gamma(K\cap \Phi(t')) < b$. The capacity $\gamma$ being lower semi-continuous (see \cite[Section 3.1]{k-l-w}), there is an open set $U$ in $X$ such that $K\cap \Phi(t')\sub U$ and $\gamma(U)<b$. There are also open sets $V_1$, $V_2$ such that $K\sub V_1$, $\Phi(t')\sub V_2$ and $V_1\cap V_2 \sub U$. Consequently, there is $n$ with $K_n\sub V_1$ and  $\Phi(t_n)\sub V_2$ from which it follows that 
$\gamma(K_n\cap \Phi(t_n)) < b$, contradicting the choice of $t_n$.
\end{proof}


\begin{remark}
It is well known that if a capacity $\gamma$ is {\sl strongly subadditive} (i.e., $\gamma(A\cup B)\leq \gamma(A)+\gamma(B)-\gamma(A\cap B)$ for $A,\ B\sub X$), then $\gamma(A)=\inf\{\gamma(U): A\sub U,\ U \hbox{ open}\}$ for all sets $A\sub X$, cf.   \cite[Th\' eor\` eme 15]{Del}. Dellacherie \cite[Appendice I, 4]{Del} with the help of the Davies' construction (cf. Section \ref{Davies_example}) gave an example of a subadditive capacity $\gamma$ on $[0,1]\times [0,1]$ and a   $\gamma$-null $G_{\delta}$-set $M$ such that $\gamma(U)=1$ for any open set $U\supseteq M$.
If in Dellacherie's example we instead use the Mazurkiewicz $G_{\delta\sigma}$-set $M$ for the projection onto the first axis $\pi:[0,1]\times [0,1]\w [0,1]$ (cf. assertion (M) in Section \ref{Mazurkiewicz}), then we still have $\gamma(M)=0$ but $\gamma(G)=1$ for any $G_\delta$-set $G\supseteq M$.
We were unable to find in the literature any examples of a subadditive capacity $\gamma$ satisfying this assertion. 
\end{remark}

\section{Uniformly tight compacta in $P(\QQ)$}\label{P(Q)}

Given a separable metrizable space $E$, we shall denote by $P(E)$ the space of probability Borel measures on $X$, equipped with the weak topology (see \cite[17.E]{k}). 

If $E$ is a  Borel set in a compactum $X$, then every measure $\mu\in P(E)$ is {\sl tight}, i.e., for every $\varepsilon>0$ there is a compact set $K\sub E$ such that $\mu(X\setminus K)<\varepsilon$ ( see \cite[Theorem 17.11]{k}).

A set $\M\sub P(E)$ is {\sl uniformly tight}, if for every $\varepsilon>0$ there is a compact set $K\sub E$ such that $\mu(X\setminus K)<\varepsilon$ for all $\mu\in\M$, cf. \cite[Definition 8.6.1]{Bo}.

A set $\M\sub P(E)$ is {\sl \s-uniformly tight}, if it is  a countable union of uniformly tight sets.

\subsection{A refinement of the Davies example}\label{Davies_example}

Let $\lambda$ be the countable product of the measure on $\{0,1\}$ assigning $\frac{1}{2}$ to each singleton and let, for $t\in\cantor$, $\lambda_t=\delta_t\otimes \lambda$ be the product of the Dirac measure at $t$ and $\lambda$, on the product  $\cantor\times\cantor$.

Davies \cite{D} considered the \s-compact set 
\begin{enumerate}
\item[(1)] $E=(\cantor\times\cantor)\setminus M,$
\end{enumerate}
where $M$ is a set described at the beginning of Section \ref{special}, i.e., $M$ is a $G_\delta$-selector for the projection $\pi:\cantor\times\cantor\w \cantor$, hitting every compact set in $\cantor\times\cantor$ projecting onto $\cantor$.

Since $\lambda_t(M)=0$ for all $t\in \cantor$, one can consider $\lambda_t$ as an element of the space $P(E)$ and 
\begin{enumerate}
\item[(2)] $\Lambda:\cantor\w P(E)$ given by $\Lambda(t)=\lambda_t$ for $t\in\cantor$,
\end{enumerate}
is a homeomorphic embedding, cf. \cite[Section 8]{T}.  

Davies pointed out that the compact set $\Lambda(\cantor)$ is not uniformly tight, as for any compact set $A$ in $E$, if $t\not\in\pi(A)$, then $\lambda_t(A)=0$, cf. also \cite{T}.

Picking a special set $M$ described in (7) and (8) of Section \ref{special}, we shall get some additional properties of this example. We shall use the notation introduced in Section \ref{special}.

Let $\C$ be the collection described in Section \ref{special}, (6), and let, cf. Section \ref{special}, (9),

\begin{enumerate}
\item[(3)] $\D=\{C\in\C: \inf\{\lambda_t(C): t\in \cantor\}>0\}$,
\item[(4)] $\M=\{\Lambda(T(C)): C\in\D  \}$.
\end{enumerate}

We shall check that the collection $\M$ of compact sets in $\Lambda(\cantor)$ has the following properties.

\begin{lemma}\label{4.1}
\begin{enumerate}
	\item Each nonempty open set in $\Lambda(\cantor)$ contains some element of $\M$.
	
	\item Whenever $A\in\M$ and $A_1, A_2,\ldots$ are uniformly tight sets in $P(E)$, there is an element of $\M$ contained in $A\setminus (A_1\cup A_2\cup\ldots)$.
\end{enumerate}
\end{lemma}

\begin{proof}
(i) Let $G$ be a nonempty open set in  $\Lambda(\cantor)$, and let $V=\Lambda^{-1}(G)$.

By Lemma \ref{T(C)}(B), thee exists $C\in \D$, cf. (3), such that $T(C)\sub V$, and hence $\Lambda(T(C))$ is an element of $\M$ contained in $G$, cf. (4).

\smallskip

(ii) Let $A\in\M$ and let, cf. (3), (4), for a certain $C\in\D$,
\begin{enumerate}
	
	\item[(5)] $\M=\Lambda(T(C))$ and $\inf\{\lambda_t(C): t\in\cantor\}=\delta>0$.
\end{enumerate}

Let $A_i\sub P(E)$ be uniformly tight.
Let us recall that if $S \sub P(E)$ is uniformly tight, then so is its closure in $P(E)$ (indeed, if $K_1 \sub K_2 \sub\ldots$ are compact sets in $E$ such that $m(K_i) \geq 1- \frac{1}{i}$ for $m \in S$, then the intersection od the closed sets $\{m \in P(E): m(K_i) \geq 1- \frac{1}{i} \}$ is a closed uniformly tight set in $P(E)$ containing $S$). 
Therefore, we can assume that $A_i$ are compact, i.e.,
\begin{enumerate}
	
	\item[(6)] $A_i=\Lambda(T_i)$ for some compact $T_i\sub\cantor$.
\end{enumerate}

Uniform compactness of $A_i$ provides a compact set 
$F_i\sub E\cap \pi^{-1}(T_i)$ such that $\lambda_t(F_i)>
1-\frac{\delta}{2^{i+1}}$, for $t\in T_i$, and one can extend $F_i$ to a compact set $H_i$ in $\cantor\times\cantor$ such that 
\begin{enumerate}
	
	\item[(7)] $H_i\cap (M\cap \pi^{-1}(T_i))=\emptyset$ and $\lambda_t(H_i)>
	1-\frac{\delta}{2^{i+1}}$, for $t\in \cantor$.
\end{enumerate}

Indeed, let us fix $F_i$ and let $U_1\supseteq U_2\ldots$ be sets open in $\cantor\times\cantor$ such that $F_i=\bigcap_n U_n$. Inductively, we pick finite unions of closed-and-open rectangles $W_i$, $W_0=\cantor\times\cantor$, such that $F_i\sub W_n\sub U_n\cap W_{n-1}$ and $\lambda_t(W_n)>1-\frac{\delta}{2^{i+1}}$, for $t\in \pi(W_n)$. Then, with $S_n=\pi(W_n)\setminus \pi(W_{n+1})$, the set $H_i=F_i\cup \bigcup_{n=0}^\infty(W_n\cap \pi^{-1}(S_n))$ has required properties.

Now, let
\begin{enumerate}
	
	\item[(8)] $H=\bigcap_{i=1}^{\infty} H_i$.
\end{enumerate}

Then $H$ is a compact set in $\cantor\times\cantor$, 
\begin{enumerate}
	
	\item[(9)] $H\cap(M\cap\pi^{-1}(\bigcup_{i=1}^{\infty} T_i))$ and
	$\lambda_t(H)>
	1-\frac{\delta}{2}$, for $t\in \cantor$.
\end{enumerate}

By (5), (9) and (3),
\begin{enumerate}
	
	\item[(10)] $B=H\cap C\in\D$.
\end{enumerate}

From Lemma \ref{T(C)}(A), $T(B)\sub T(C)\setminus \bigcup_{i=1}^{\infty} T_i$, and in effect, by (10), (5) and (6), $\Lambda(T(B))$ is an element of $\M$ contained in $A\setminus \bigcup_{i=1}^{\infty} A_i$.
\end{proof}

\subsection{Transferring the Davies example into $P(\QQ)$: a proof of Theorem\ref{main}}\label{proof_of_main}

Let $f:E\w F$ be a perfect map from a separable metrizable space $E$ onto a closed subspace of a separable metrizable space $F$. The map $f$ gives rise to a perfect map $P(f): P(E)\w P(F)$ defined by $P(f)(\mu)= \mu\circ f^{-1}$, such that $P(f)(A)$ (respectively,  $P(f)^{-1}(B)$) is uniformly tight, whenever   $A$ (respectively, $B$) is uniformly tight, cf. \cite[Theorem 8.10.30]{Bo}.

Let $E=(\cantor\times\cantor)\setminus M$ be the Davies' example discussed in Section \ref{Davies_example}. Saint Raymond \cite{SR} defined a perfect map $f:E\w \QQ$ (cf. also \cite{Ju} and \cite{Ost}) and concluded that $P(\QQ)$ contains a compact non-uniformly tight set, thus providing a proof of the Preiss theorem, based on different ideas than the original one. We shall use this approach to prove Theorem \ref{main}. More precisely, we shall use the special set $M$, discussed in Section \ref{special}, and we shall appeal to the following observation.

\begin{lemma}\label{reduction}
Let $G$ be a $G_\delta$-set in $\cantor$. There is a continuous map $p:\cantor\w\cantor$such that
\begin{enumerate}
	\item[(i)] $p$ embeds $G$ homeomorphically into $\cantor\setminus\QQ$ and maps $\cantor\setminus G$ into $\QQ$,
	
	\item[(ii)] for any disjoint compact sets $A, B$ in $\cantor$, $p(A)\cap p(B)$ is finite.
\end{enumerate}
\end{lemma}

\begin{proof}
We shall use an idea similar to that in \cite[proof of Lemma]{E-P}.

Let us fix a metric on $\cantor$, and let 
\begin{enumerate}
	\item[(11)] $G=H_1\cap H_2\cap\ldots$, where $H_1\supseteq H_2\supseteq\ldots$ are open in $\cantor$.
\end{enumerate}

Let us split each $H_n$ into pairwise-disjoint closed-and-open sets $V_{n,1}, V_{n,2},\ldots$ such that
\begin{enumerate}
	\item[(12)] $H_n=\bigcup_i V_{n,i}$, $\hbox{diam} V_{n,i}\leq \frac{1}{n}$ and $\lim_i \hbox{diam} V_{n,i} = 0 $.
\end{enumerate}

Let $e_0$ be the zero sequence in $\cantor$ and let $e_i$ have exactly one non-zero coordinate, at the $i$'th place.

The function $p_n:\cantor\w\cantor$ sending $\cantor\setminus H_n$ to $e_0$ and $V_{n,i}$ to $e_i$, is continuous, and let
\begin{enumerate}
	\item[(13)] $p=(p_1, p_2,\ldots):\cantor\w \cantor\times\cantor\times\ldots$ 
\end{enumerate}
be the diagonal map. Fixing a bijection between $\NN\times\NN$ and $\NN$, we shall identify $\cantor\times\cantor\times\ldots$ with $\cantor$.

Clearly, $p$ satisfies (i).

To check (ii), let us consider disjoint compact sets $A,\ B$ in $\cantor$ and let $\delta>0$ be the distance between $A$ and $B$. By (12), for a fixed $n$, only finitely many $V_{n,i}$'s intersect both $A$ and $B$, and therefore $p_n (A)\cap p_n (B)$ is finite. Moreover, if $\frac{1}{n}<\delta$, no $V_{n,i}$ intersects both $A$ and $B$, hence $p_n(A)\cap p_n(B)\sub \{e_0\}$.

It follows that, cf (13),  $p(A)\cap p(B)$ is a subset of some product $\cantor\times\dots\cantor\times \{e_0\}\times \{e_0\}\times \ldots$ whose every projection is finite.

\end{proof}	

We are now ready to justify Theorem \ref{main}.

Let $E$ be the space discussed in Section \ref{special}, and let $p:\cantor\times\cantor\w \cantor$ be the map described in Lemma \ref{reduction} for $G=M$. The map
\begin{enumerate}
\item[(14)] $f=p|E:E\w \QQ$ is perfect, 
\end{enumerate}
and let
\begin{enumerate}
\item[(15)] $P(f):P(E)\w P(\QQ)$ be defined by $P(f)(\mu)=\mu\circ f^{-1}$. 
\end{enumerate}

We shall check that the compact set
\begin{enumerate}
\item[(16)] $K=P(f)(\Lambda(\cantor))\sub P(\QQ)$ 
\end{enumerate}
has the properties (i) and (ii) in Theorem \ref{main}.

For any $t\in\cantor$, the support of the measure $P(f)(\lambda_t)$ is the set $f((\{t\}\times\cantor)\setminus M)$, and from (i) and (ii) in Lemma \ref{reduction} we obtain property (i) in Theorem \ref{main}.

Let $\N=\{P(f)(A):\ A\in\M) \}$. Then Lemma \ref{reduction} implies that each non-empty open set in $K$ contains an element of $\N$ and for each $S\in \N$ and uniformly tight sets $S_1, S_2,\ldots$ in $P(\QQ)$, there is an element of $\N$ contained in $S\setminus (S_1\cup S_2\cup\ldots)$. This yields (ii) in Theorem \ref{main}.

\section{Borel mappings on $P(\QQ)$}\label{1-1}


A reasoning in \cite{PZ} can be used to the following effect.

\begin{proposition}\label{constant_or_injective}
	Let $\K$ be a hereditary collection of compact sets in a compactum $K$ such that
	\begin{enumerate}	
	\item[(i)] 	for any non-empty open set $V$ in $K$ there is a compact set $A\sub V$ such that, whenever $A_1, A_2,\ldots\in\K$, there is a compactum not in $\K$, contained in $A\setminus (A_1\cup A_2\cup\ldots)$.
		\item[(ii)] no compact set $A\notin\K$ can be covered by countably many elements of $\K$. 
	\end{enumerate}
		  Then any Borel map $f: K\w [0,1]$ is either constant or injective on a Borel set in $K$ which cannot be covered by countably many elements of $\K$.
\end{proposition}

	\begin{proof} Let $I$ be the \s-ideal in $K$ generated by sets in $\K$, i.e., $I$ consists of Borel sets which can be covered by countably many elements of $\K$. Let us note that no open set in $K$ belongs to $I$.
		
	We shall derive the proposition from the following claim.
	
	\begin{claim}
		For any Borel map $f: K\w [0,1]$ there is a compact meager set $C$ in $[0,1]$ with $f^{-1}(C)\not\in I$.
	\end{claim}	
			
	To prove the claim, we shall repeat the reasoning from Section 3 of \cite{PZ}. To keep the notation close to that in \cite{PZ}, we let $X= K$, $Y=[0,1]$, and striving for a contradiction, let us assume that for any meager set $C$ in $Y$, $f^{-1}(C)\in I$.
	
	There exists a $G_\delta$-set $G$ in $X$, dense in $X$ such that $f|G: G\w Y$ is continuous. Since every compact set in $I$ must have empty interior, 
	$V\not\in\I$ for any nonempty relatively open $V$ in $G$. 
	Thus, (1) and (2) in Section 3 of \cite{PZ} are satisfied. 
	
	Let us check that the assertion of Claim 3.1 of \cite{PZ} holds true in our situation. This requires a minor modification of the arguments.
	
	In Case 1, i.e.,  if the set $\widecheck{f}_U(d)$ is not in $I$, then either it is boundary (and then we can take $L=\widecheck{f}_U(d)$) or otherwise, by (i),  it contains a compact set $A\not\in I$ (and then we can just take $L=A$). 
	
	In Case 2, i.e., if $\widecheck{f}_U(d)\in I$ for all $d\in D$, then, $\overline{U}$ having non-empty interior, using (i) we find a  boundary compactum 
	$L\sub \overline{U}\setminus \bigcup_{d\in D}\widecheck{f}_U(d)$, $L\notin \K$ so, consequently, $L\notin I$ by (ii).
	
	The rest of the proof in Section 3 in \cite{PZ} does not require any change, and we reach in this way a contradiction ending the proof of the claim.
	
	\smallskip
	
	Combining the claim with the reasoning leading to \cite[Theorem 3.2]{PZ-1} we get the assertion of Proposition \ref{constant_or_injective}.
	
	 			\end{proof}
	
	
	We would like to apply Proposition \ref{constant_or_injective} to the compactum $K$ defined in Theorem \ref{main} and to the collection $\K$ of compact uniformly tight sets in $K$ to the following effect:
	\begin{itemize}
		\item any Borel map $f: K\w [0,1]$ is either constant or injective on a Borel non-\s-uniformly tight set in $K$. 
	\end{itemize}
		
	 In view of Theorem \ref{main}(ii), it is enough to check that $\K$ satisfies assertion (ii) of Proposition \ref{constant_or_injective}. The latter will be an immediate consequence of a
	  result we are about to prove in a more general setting.
	 
	Given a separable metrizable space $E$, we shall denote by $P_t(E)$ the space of tight probability Borel measures on $E$, equipped with the weak topology; if $E$ is a Borel subset of a separable, completely metrizable space, then $P_t(E)=P(E)$ -- the space of all probability Borel measures on $E$, cf. Section \ref{Davies_example}. 
	 
	The following result extends a theorem of Hoffman-J\o  rgensen \cite{H-J} and Choquet \cite{Ch} that countable compact sets in $P_t(E)$ are uniformly tight (which in turn generalized the classical Le Cam theorem about convergent sequences in $P_t(E)$, cf. \cite[Theorem 8.6.4]{Bo}). 
	Its proof is rather standard but we did not find a handy reference in the literature, so we decided to include a proof for readers' convenience.
	
	\begin{proposition}\label{sigma_ideal}
	Let $L$ be a compact set of tight probability Borel measures on a separable metrizable space $E$. 
	If $L$ is a countable union of compact uniformly tight sets, then $L$  is uniformly tight.
	\end{proposition}

We shall derive this result from the following lemma.

\begin{lemma}\label{lemma_for_sigma_ideal}
	Let $L\sub P_t(E)$ be a compact set such that for some compact $A\sub L$ which is uniformly tight, all compact sets in $L$ disjoint  from $A$ are uniformly tight. Then $L$ is uniformly tight.
\end{lemma}
	
\begin{proof}
Let $\varepsilon>0$, and let $C\sub E$ be a compact set such that 
	
\begin{enumerate}
	\item[(1)] $\mu(C)> 1- \frac{\varepsilon}{8}$ for all $\mu\in A$. 
\end{enumerate}	

The key observation is the following

\begin{claim}
	If $U$ is an open set in $E$ containing $C$, then there exists a relatively open subset $W$ of $L$ containing $A$ such that $\nu(U) >1- \frac{\varepsilon}{4}$ for all $\nu\in W$. 
\end{claim}
	
To justify the claim, for each $\mu\in A$, let us pick an open subset $U_\mu$ of $E$ such that  $C\sub U_\mu\sub\overline{U_\mu}\sub U$ and $\mu(\overline{U_\mu}\setminus U_\mu)=0$. Then
$$
W_\mu=\Bigl\{\nu\in P_t(E): |\nu(U_\mu)- \mu(U_\mu)|<\frac{\varepsilon}{8}  \Bigr\}
$$
is an open neighbourhood of $\mu$ in $P_t(E)$.

By compactness of $A$, there are $\mu_1,\ldots, \mu_k\in A$ such that $W_{\mu_1},\ldots,W_{\mu_k}$ cover $A$, and let 
$$
W=L\cap (W_{\mu_1}\cup\ldots\cup W_{\mu_k} ).
$$ 

If $\nu\in W$, then $\nu\in W_{\mu_i}$ for some $i\leq k$, and then, cf. (1),
$$
\nu(U)\geq \nu(U_{\mu_i})\geq \mu_i(U_{\mu_i})- \frac{\varepsilon}{8} > \mu_i(C)- \frac{\varepsilon}{8} > 1 - \frac{\varepsilon}{4}
$$
which ends the proof of the claim.

\smallskip

Using this observation one can inductively define open subsets $U_1, U_2,\ldots$ of $E$ such that
\begin{enumerate}
	\item[(2)] $U_1\supseteq \overline{U_2}\supseteq  U_2\supseteq \overline{U_3}\supseteq U_3\ldots \supseteq C$,
\end{enumerate}
and relatively open in $L$ sets
  \begin{enumerate}
  	\item[(3)] $W_1\supseteq  W_2\supseteq  W_3\ldots \supseteq A$,
  \end{enumerate}
such that 
 \begin{enumerate}
	\item[(4)] $dist(C, E\setminus U_k)\to 0,\ \bigcap_i W_i = A$,
\end{enumerate}
and
 \begin{enumerate}
	\item[(5)] $\mu(U_i)> 1 - \frac{\varepsilon}{4} $ for all $\mu\in W_i$.
\end{enumerate}

By the assumption,
 \begin{enumerate}
	\item[(6)] $B_i=L\setminus W_i $ is uniformly tight, $i=1,2,\ldots$
\end{enumerate}
and let $D_i\sub E$ be a compact set such that
 \begin{enumerate}
	\item[(7)] $\mu(D_i)> 1 - \frac{\varepsilon}{4} $ for all $\mu\in B_i$, $i=1,2,\ldots$.
\end{enumerate}

By (4), the set
\begin{enumerate}
	\item[(8)] $D=C\cup D_1\cup \bigcup_{i=1}^\infty \overline{D_{i+1}\cap U_{i}}$ 
\end{enumerate}
is compact.

Let $\mu\in L\setminus A$. 

If $\mu\not\in W_1$, then $\mu(D)\geq \mu(D_1)> 1 - \frac{\varepsilon}{4}$, by (6) and (7).

If $\mu\in W_1$, let us pick $i$ such that $\mu\in W_i\setminus W_{i+1}$, cf. (4). Then $\mu\in B_{i+1}$, cf. (6), and by (5) and (7), $\mu(E\setminus U_i)<\frac{\varepsilon}{4}$ and $\mu(D_{i+1})> 1 - \frac{\varepsilon}{4}$. It follows that
$$
\mu(\overline{D_{i+1}\cap U_{i}})\geq \mu(D_{i+1})-\mu(E\setminus U_i)> 1 - \frac{\varepsilon}{2}.
$$

In effect, by (1) and (8), $\mu(D)> 1 - \frac{\varepsilon}{2}$ for all $\mu\in L$, completing the proof of Lemma \ref{lemma_for_sigma_ideal}.

\end{proof}

We are now ready to complete the proof of Proposition \ref{sigma_ideal}.

\begin{proof}[Proof of Proposition \ref{sigma_ideal}]
Let a compact set $L\sub P_t(E)$ be a countable union of compact uniformly tight sets. 


 One can inductively construct a transfinite sequence of compact sets 
$$
F_0=L\supseteq F_1\supseteq\ldots\supseteq F_\xi\supseteq\ldots\supseteq F_\alpha
$$ 	
where $0\leq\alpha<\omega_1$, such that 
\begin{enumerate}
\item[(9)] $F_0=L$, 
	\item[(10)] $\overline{F_\xi\setminus F_{\xi+1}}$ is uniformly tight for $\xi<\alpha$,
		\item[(11)] $F_\lambda=\bigcap_{\xi<\lambda}F_\xi$ for limit $\lambda\leq \alpha$,
		\item[(12)] $F_\alpha$ is uniformly tight.
		
\end{enumerate}	

We start with $F_0=L$.

At the successor step, if $F_\xi$ is uniformly tight, then we complete the construction by letting $\alpha=\xi$.
Otherwise,  since $F_\xi$ is covered by countably many  uniformly tight compacta, the Baire category theorem yields a relatively open  set $U_\xi\subsetneq F_\xi$ whose closure is uniformly tight, and we let $F_{\xi+1}=F_\xi \setminus U_\xi$.


We shall now check by induction on $\alpha<\omega_1$ the following fact. 

\begin{claim}
	If $\alpha<\omega_1$ and a compactum  $L\sub P_t(E)$ admits a sequence $(F_\xi)_{\xi\leq\alpha}$ satisfying conditions (9)--(12), then $L$ is uniformly tight. 
\end{claim}

If $\alpha=0$, then there is nothing to do, so let us assume that $\alpha>0$ and  the claim holds true for all $\beta<\alpha$. Let  $(F_\xi)_{\xi\leq\alpha}$ be a sequence satisfying conditions (9)--(12) for a given compactum $L\sub P_t(E)$. To prove that $L$ is uniformly tight, it suffices to check that each compact set in $L$ disjoint from $F_\alpha$ is uniformly tight, and then Lemma \ref{lemma_for_sigma_ideal} provides readily the assertion.

So let $S$ be any compact subset of $L\setminus F_\alpha$. Let us consider two cases.

\smallskip

{\sl Case 1.} $\alpha$ is a limit ordinal. Then 
$F_\alpha=\bigcap_{\xi<\alpha}F_\xi$, cf. (11). By compactness, there is $\beta<\alpha$ such that $S\cap F_\beta=\emptyset$. 

\smallskip

{\sl Case 2.} $\alpha=\beta+1$. Then, by (10) and (12), $F_\beta$ is uniformly tight.

\smallskip

In each case $(F_\xi\cap S)_{\xi\leq\beta}$ witnesses that $S$ admits a shorter sequence  satisfying conditions (9)--(12), so by the inductive assumption, $S$ is uniformly tight.


\end{proof}

\section{Comments}\label{Comments}




\subsection{The \s-ideal of uniformly tight sets in $P(\QQ)$}

(A) We say that a \s-ideal $I$ of compact sets in a compactum $X$ is {\sl calibrated} (cf. \cite{k-l-w}) 
if for any compact set $A\not\in I$, 

\begin{enumerate}
	\item[$(*)$] whenever $A_1, A_2,\ldots\in I$, there is a compactum not in $I$, contained in $A\setminus (A_1\cup A_2\cup\ldots)$.
\end{enumerate}

Let $K$ be the compactum in  $P(\QQ)$ constructed in the proof of Theorem \ref{main} and let $I_{ut}(K)$ be the collection of compact uniformly tight sets in $K$. By Proposition \ref{sigma_ideal}, 
$I_{ut(K)}$ is a \s-ideal of  compact sets in $K$, cf. Section \ref{thin}. Moreover, if we let $J_{ut}(K)$ be the collection of all Borel uniformly tight sets in $K$, then 
$J_{ut}$ is a \s-ideal on $K$, generated by compact sets (the latter follows from the fact that the closure of a uniformly tight set is always a compact uniformly tight set, which constitutes a part of the Prokhorov theorem, cf. \cite[Theorem 6.7]{Par}

Theorem \ref{main}(ii) combined with Proposition \ref{sigma_ideal} show that every open set in $K$ contains 
a compact set $A\not\in I_{ut}(K)$ with property $(*)$ for $I=I_{ut}(K)$. However, we do not know if  $I_{ut}(K)$ is calibrated. If this were indeed the case we would have the ``1-1 or constant" property for Borel sets in $K$, cf. \cite{SZ}, which would considerably strengthen the assertion formulated in Section \ref{1-1} just before Proposition \ref{sigma_ideal}.

\smallskip

Let us notice that the reasoning in Section \ref{P(Q)} shows in fact that, in the following game involving two players, the second player always has 
a winning strategy:
the first player chooses compact uniformly tight sets $A_1, A_2,\ldots$ in $P(\QQ)$, the response of the second player to the move $A_i$ of the first player is a compact set $K_i$ in $P(\QQ)$ disjoint from $A_i$, and the second player wins if $\bigcap_i K_i$ is not uniformly tight.  

\medskip

(B) Refining the construction in Section \ref{P(Q)} one can show that the \s-ideal $I_{ut}(K)$ in not analytic. We do not know, what is the exact descriptive complexity of this \s-ideal (in particular, whether it is coanalytic).

\subsection{The supports of measures}

Given a separable metrizable space $E$, a theorem of Balkema \cite{Ba} (cf. also \cite{T}) asserts that any compact set $K$ in $P_t(E)$ whose elements have compact supports, is uniformly tight.

The measures in the non-uniformly tight  compactum $K$ in  $P(\QQ)$, constructed in the proof of Theorem \ref{main}, have locally compact supports. Moreover, identifying (as we did) $\QQ$ with the set of points in $\cantor$ with finite supports, one can see that for any $\mu\in K$, the closure $\overline{\supp(\mu)}$ of $\supp(\mu)$ in $\cantor$ adds at most one point and, in particular, $\overline{\supp(\mu)}$ is a compact scattered subset of $\cantor$. Consequently, the collection $\A=\{\overline{\supp(\mu)}:\mu\in K \}$ is an analytic set in $K(\cantor)$ consisting of scattered sets, and  by a classical Hurewicz's theorem \cite{hur}, there is $\alpha<\omega_1$ such that for any $\mu\in K$, the Cantor-Bendixson index of $\supp(\mu)$ is bounded by $\alpha$. We do not know, what is the minimal possible bound $\alpha$ in this situation.

\subsection{Any Mazurkiewicz's function is strongly non-\s-continuous}

Let us recall that a function $f: T\w \cantor$ on a subset $T$ of $\cantor$ is {\sl \s-continuous}, if $T$ can be decomposed into countably many sets $T_i$ such that each restriction $f|T_i: T_i\w \cantor$ is continuous, cf. \cite{sol}, \cite{PS}.

Let $F\sub\cantor\times\cantor$ be a compact set with the following property: for each compact $C\sub \cantor\times\cantor$ with $\pi(C)=\cantor$ there is $t\in\pi(F)$ such that $(\{t\}\times\cantor)\cap F\sub C$ (such ``diagonal" compacta $F$ appear in the Mazurkiewicz construction discussed in Section 2).

Then for any $f:\pi(F)\w \cantor$ whose graph is contained in $F$ and every perfect set $L\sub\cantor$, the restriction $f|f^{-1}(L)$ is not \s-continuous.

This can be verified by a reasoning similar to that used in the proof of property (2) in Lemma \ref{4.1}.

\bibliographystyle{amsplain}

\end{document}